\documentclass[10pt]{article}

\usepackage[utf8]{inputenc}
\usepackage[T1]{fontenc}

\usepackage{epsf}
\usepackage{amsmath}

\allowdisplaybreaks

\usepackage[showframe=false]{geometry}
\usepackage{changepage}

\usepackage{epsfig}
\usepackage{amssymb}

\usepackage{amsthm}
\usepackage{setspace}
\usepackage{cite}
\usepackage{mcite}

\usepackage{algorithmic}  
\usepackage{algorithm}

\usepackage{shadow}
\usepackage{fancybox}
\usepackage{fancyhdr}

\usepackage{color}
\usepackage[usenames,dvipsnames,svgnames,table]{xcolor}
\newcommand{\bl}[1]{\textcolor{blue}{#1}}

\definecolor{mypurple}{rgb}{.4,.0,.5}

\usepackage[hyphens]{url}

\usepackage[colorlinks=true,
            linkcolor=black,
            urlcolor=blue,
            citecolor=purple]{hyperref}

\usepackage{breakurl}

\def\x{{\bf x}}

\def\x{{\mathbf x}}

\def\x{{\bf x}}

\def\be{\begin{equation}}
\def\ee{\end{equation}}
\def\ba{\left[\begin{array}}
\def\ea{\end{array}\right]}

\def\x{{\bf x}}

\def\1{{\bf 1}}

\def\g{{\bf g}}
\def\0{{\bf 0}}

\def\erfc{\mbox{erfc}}

\def\mR{{\mathbb R}}
\def\mN{{\mathbb N}}
\def\mE{{\mathbb E}}
\def\mS{{\mathbb S}}
\def\mB{{\mathbb B}}
\def\mP{{\mathbb P}}

\def\lp{\left (}
\def\rp{\right )}

\def\x{{\bf x}}

\def\x{{\mathbf x}}

\def\x{{\bf x}}

\def\be{\begin{equation}}
\def\ee{\end{equation}}
\def\ba{\left[\begin{array}}
\def\ea{\end{array}\right]}

\def\x{{\bf x}}

\def\({\left (}
\def\){\right )}

\def\1{{\bf 1}}

\def\g{{\bf g}}
\def\0{{\bf 0}}

\def\cS{{\mathcal S}}
\def\cN{{\mathcal N}}
\def\cG{{\mathcal G}}
\def\bcS{\bar{\mathcal S}}

\usepackage{xcolor}
\usepackage{color}

\definecolor{darkgreen}{rgb}{0, 0.4,0}

\definecolor{purplebrown}{rgb}{0.5,0.1,0.6}

\definecolor{ultclupcol}{rgb}{0.1,0.5,0.5}

\definecolor{mytrycolor}{rgb}{0.5,0.7,0.2}

\definecolor{ultclupcola}{rgb}{.5,0,.5}

\definecolor{shadebrown}{rgb}{0.1,0.1,0.9}
\definecolor{lightblue}{rgb}{0.2,0,1}

\usepackage{fancybox}
\usepackage{graphicx}
\usepackage{epstopdf}
\usepackage{epsfig}
\usepackage{wrapfig}
\usepackage{subfigure}

\usepackage{xcolor}
\usepackage{tcolorbox}

\newtcbox{\xmybox}{on line,
arc=7pt,
before upper={\rule[-3pt]{0pt}{10pt}},boxrule=0pt,
boxsep=0pt,left=6pt,right=6pt,top=0pt,bottom=0pt,enhanced, coltext=blue, colback=white!10!yellow}

\newtcbox{\xmyboxa}{on line,
arc=7pt,
before upper={\rule[-3pt]{0pt}{10pt}},boxrule=0pt,
boxsep=0pt,left=6pt,right=6pt,top=0pt,bottom=0pt,enhanced, colback=white!10!yellow}

\newtcbox{\xmyboxb}{on line,
arc=7pt,
before upper={\rule[-3pt]{0pt}{10pt}},boxrule=1pt,colframe=darkgreen!100!blue,
boxsep=0pt,left=6pt,right=6pt,top=0pt,bottom=0pt,enhanced, colback=white!10!yellow}

\newtcbox{\xmyboxc}{on line,
arc=7pt,
before upper={\rule[-3pt]{0pt}{10pt}},boxrule=.7pt,colframe=blue!100!blue,
boxsep=0pt,left=6pt,right=6pt,top=0pt,bottom=0pt,enhanced, coltext=blue, colback=white!10!yellow}

\newtcbox{\xmytboxa}{on line,
arc=7pt,
before upper={\rule[-3pt]{0pt}{10pt}},boxrule=.0pt,colframe=pink!50!yellow,
boxsep=0pt,left=6pt,right=6pt,top=0pt,bottom=0pt,enhanced, coltext=white, colback=blue!40!red}

\newtcbox{\xmytboxb}{on line,
arc=7pt,
before upper={\rule[-3pt]{0pt}{10pt}},boxrule=.0pt,colframe=pink!50!yellow,
boxsep=0pt,left=6pt,right=6pt,top=0pt,bottom=0pt,enhanced, coltext=white, colback=white!40!green}

\makeatletter
\newcommand\subsubsubsection{\@startsection{paragraph}{4}{\z@}{-2.5ex\@plus -1ex \@minus -.25ex}{1.25ex \@plus .25ex}{\normalfont\normalsize\bfseries}}
\newcommand\subsubsubsubsection{\@startsection{subparagraph}{5}{\z@}{-2.5ex\@plus -1ex \@minus -.25ex}{1.25ex \@plus .25ex}{\normalfont\normalsize\bfseries}}
\makeatother

\newtheorem{theorem}{Theorem}

\newtheorem{corollary}{Corollary}

\newtheorem{remark}{Remark}

\begin{document}

\begin{singlespace}

\title {Ground state energies of multipartite $p$-spin models  -- partially lifted RDT view 
}
\author{
\textsc{Mihailo Stojnic
\footnote{e-mail: {\tt flatoyer@gmail.com}} }}
\date{}
\maketitle

\centerline{{\bf Abstract}} \vspace*{0.1in}

We consider ground state energies (GSE) of multipartite $p$-spin models. Relying on partially lifted random duality theory (pl RDT) concepts we introduce an analytical mechanism that produces easy to compute lower and upper GSE bounds for \emph{any} spin sets. We uncover that these bounds actually match in case of fully spherical sets thereby providing optimal GSE values for spherical multipartite pure $p$-spin models. Numerical evidence further suggests that our upper and lower bounds may match even in the Ising scenarios. As such developments are rather intriguing, we formulate several questions regarding the connection between our bounds matching generality on the one side and the spin sets structures on the other.

\vspace*{0.25in} \noindent {\bf Index Terms: Multipartite p-spin models; Ground state energy; Random duality theory}.

\end{singlespace}

\section{Introduction}
\label{sec:back}

In \cite{SheKir72} the famous Sherrington-Kirkpatrick (SK) model was introduced as a long range antipode to the  Edwards-Anderson (EW) nearest neighbor  corresponding one from  \cite{EdwAnd75}.   After early analytical \cite{SheKir72} and computational \cite{SheKir78} recognitions that the so-called \emph{replica-symmetry} (RS) ansatz of \cite{EdwAnd75} is unsustainable, Parisi's \emph{replica-symmetry breaking} (RSB) refinement  \cite{Par79,Parisi80,Par80} turned out to be a monumental discovery that would go on to shape the research in a variety of different fields including pure and applied mathematics, information theory, machine learning, and algorithmic computer science. While mathematically rigorous confirmations of various replica predictions slowly started appearing in the eighties of the last century, it took close to 25 years until Guerra \cite{Guerra03} and Talagrand \cite{TalSK06} proved the original Parisi SK RSB formalism. Soon thereafter Panchenko  \cite{Pan10,Pan10a,Pan13,Pan13a} reproved these results while additionally establishing the \emph{ultrametricity} \cite{Par83}. Relying on the so-called A.S.S. scheme \cite{Aizetal03} and Ghirlanda-Guerra identities \cite{Pan10,Pan10a} he developed methods that were sufficiently powerful to handle mixed $p$-spin SK models for any $p$ thereby complementing  the original Talagrand 's proof for pure even $p$-spin ones (Ising $p$-spin models were considered within statistical physics by Derrida in \cite{Derrida80}; see also, \cite{Derrida81}). In \cite{ChenSphSK13}, Chen further extended Panchenko's Ising SK considerations to encompass spherical mixed $p$-spin ones as well (this complemented Talagrand's spherical pure even $p$-spin results from \cite{TalSph06}; for corresponding statistical physics considerations of $p$-spin spherical models see, e.g., \cite{KostSKSph76,Crisanti92}). Due to their practical importance/relevance and a direct connection to optimization and algorithms, the so-called ground state energy (GSE) regimes were often treated separately. For example, variants of the Parisi's formulas have been developed particularly for such regimes for both, the Ising  \cite{AuffChen17}  and the spherical SK models  \cite{JagTob17,ChenSen17}.

All of the above developments triggered an avalanche of excellent followup results that further uncovered many beautiful mathematical and spin glasses properties \cite{Talbook11a,Talbook11b,AuffChen15,Auffetal20,JagTob16,MourratPan20,HuangSell23a,Gametal23} as well as a large set of remarkable algorithmic consequences \cite{Gamar21,GamJag21,AlaouiMS22,AlaouiMS25,HuangS22,HuangSell23,HuangSell24,Subag24,Subag21,Montanari19}. Along the same lines, over the last 15 years many extensions of classical Ising and spherical SK models were considered as well \cite{ChenMourrat25,Mourrat23,PanMult15,PanPotts18,ChenIM25,Chenetal21,Chenetal23,Subag23,ChenPan18,Barraetal15,Ko20,DeyWu21}.
 Of our particular interest are the so-called bipartite  or (more generally) multipartite $p$-spin models \cite{BaikLee20,Barra2011,McKenna24,Kivimae23,Little74,Hop82,Pan18,MourratInf24,StojnicMoreSophHopBnds10,StojnicAsymmLittBnds11,AuffChen14,Amari72,AGS87,AmiGutSom85,Barraetal18,ContGal08,Fedele12,Albetal21}. These types of models are directly connected to statistical physics \cite{KorFyodShen87,KorFyodShen87a,KorShen85}, cognitive learning \cite{Little74,Hop82,KroHop16,Amari72,Hebb49,PasFig78}, random tensors \cite{Bandetal25,DarMc24,Rosetal19,BAetal19,ChenHL21,Lesetal17,JagLM20},  quantum entanglement  \cite{Shim95,BL01,WG03},  and machine learning and neural networks concepts \cite{DeanRit01,AGS87,AmiGutSom85,Barraetal18,Rametal21,Stojnichebhop24,Newman88,MPRV87,FST00,Louk94,Louk97,PasShchTir94,ShchTir93,BarGenGueTan10,BarGenGueTan12,BovGay98,Tal98}. A nice introductory overview of the single partite, bipartite, and multipartite models together with direct connections to some of the most relevant application scenarios can be found in \cite{MourratInf24}. Throughout  recent literature models of this type have been often studied within a wider -- so-called \emph{multi-species} -- group of models  \cite{FedCon11,ChenIM25,Pan18,Barraetal15,Batesetal19,BatesSohn22a,BatesSohn22,Subagtap23a,Subagtap23b,BatesSohn25}. As in single-species scenarios, two particular spin sets, Ising and spherical, are the standard benchmarks with the Ising variants being  analytically much harder. To handle Ising multi-species models presence of the convexity of certain overlap functionals is usually required \cite{Pan18,Barraetal15} (under similar convexity assumptions,  spherical counterparts were handled in \cite{Batesetal19,BatesSohn22a,BatesSohn22} as well). On the other hand, specialization from multi-species to multipartite models disallows the needed convexity. While in such contexts results of \cite{Pan18} maintain a lower-bounding character, the bounds they produce often become very loose. Several notable exceptions appeared very recently. \cite{BatesSohn25}  studied so-called balanced multi-species models where solid lower bounds were obtained that in certain cases (high-temperature regimes, convex overlap functionals, certain spherical regimes) even match the known upper ones (see also \cite{Issa24} for related results in the context of vector spin models). \cite{DarMc24} connected studying tensors injective norms \cite{Groth52,Lim21}  and geometric quantum entanglement \cite{Shim95,BL01,WG03}   and, through a utilization of spectral methods and analysis of critical points complexity, obtained upper bounds that correspond to the spherical lower bounds of \cite{BatesSohn25}.  On the other hand, relying on a TAP equations based approach, Subag in \cite{Subagtap23a} developed a general method for handling spherical multi-species models and in \cite{Subagtap23b} specialized it to pure models (for related general TAP approaches, see also, e.g., \cite{Chenetal23,Chenetal21}). In particular, Subag's methods circumvent the above mentioned convexity requirement and thereby allow automatic translation to multipartite spherical mixed $p$-spin models. However, they still do require convergence properties of certain energy mixtures.

 Away from spherical scenarios multipartite models remain notoriously difficult and getting even solid bounds is not an easy task. For bipartite models upper/lower bounds are given in \cite{StojnicAsymmLittBnds11} (they are given for the Ising spin sets but the presented concepts automatically apply to any sets). The lower bounds from  \cite{StojnicAsymmLittBnds11} match the RSB predictions of Brunetti, Parisi, and Ritort \cite{BruParRit92}. In particular, studying the so-called (asymmetric) Little model via replica methods, \cite{BruParRit92} postulated that it should essentially resemble the behavior of the (sum of two) SK Ising models. Lower-bounding effect of such a formulation was basically rigorously confirmed in \cite{StojnicAsymmLittBnds11}.

 Moving to multipartite scenario makes things even harder. While it is simple to rewrite  multipartite variant of fully lifted random duality theory (fl RDT) \cite{Stojnicflrdt23}, its practical utilizations eventually heavily rely on numerical evaluations. Since the parametric complexity  grows exponentially as the number of spin sets increases, complexity of conducting all underlying evaluations may become prohibitively large already for multipartite systems with $p\sim 5-10$ spin sets. This effectively poses as an imperative search for computationally more efficient characterizations and we provide such a characterization in this paper. The method that we present corresponds to partially lifted (pl) RDT and produces much easier to compute both lower and upper multipartite pure $p$-spin GSE bounds for any type of spin sets. We uncover that these bounds actually match in case of fully spherical sets thereby providing optimal GSE values. Furthermore, we show that optimal GSE values match the above mentioned conditional prediction of \cite{Subagtap23b}, the upper bound of \cite{DarMc24}, and the lower bound of \cite{BatesSohn25}. We show that the numerical agreement between results od \cite{Subagtap23b} and \cite{DarMc24} observed in \cite{DarMc24} is actually fully analytical as well. Moreover, a strong numerical evidence suggests that our lower and upper bounds may match even in the Ising scenarios. Such developments are rather intriguing and allow to draw a  parallel with  \cite{Stojnicl1RegPosasym} where  the compressed sensing large deviations were settled through a consideration of more general convex sets. Moreover, we utilize such a parallel  towards the end of the paper to formulate several questions regarding universality of spin sets for which the matching of the presented bounds happens.

\section{Main results}
\label{sec:res}

We start by recalling that for an integer $n\in\mN$, standard normal matrix $A\in\mR^{n\times n}$, and set $\mB^n \triangleq \left \{ \x| \x\in\mR^n, \x_i^2=\frac{1}{n}\right \}$ (the vertices of the $n$-dimensional binary cube),  the SK's ground state energy is given as
\begin{eqnarray}\label{eq:inteq1}
\xi^0_{sk} = \lim_{n\rightarrow\infty} \frac{1}{\sqrt{n}}\mE_A \max_{\x\in\mB^n} \x^TA\x.
\end{eqnarray}
The spherical analogue is
\begin{eqnarray}\label{eq:inteq1a}
\xi^0_{sph} = \lim_{n\rightarrow\infty} \frac{1}{\sqrt{n}}\mE_A \max_{\x\in\mS^n} \x^TA\x,
\end{eqnarray}
where $\mS^n=\{\x|\|\x\|_2=1\}$, is the unit sphere in $\mR^n$. One then easily recognizes a generic analogue
\begin{eqnarray}\label{eq:inteq1b}
\xi^0(\cS) = \lim_{n\rightarrow\infty} \frac{1}{\sqrt{n}}\mE_A \max_{\x\in\cS} \x^TA\x,
\end{eqnarray}
where $\cS\subseteq\mS^n$ (as will be clear throughout the presentation, restricting $\cS$ to the unit sphere is not needed but will make the writing way more elegant; along the same lines, to further lighten the presentation, the subscript next to $\mE$ indicating the source of randomness will often be omitted as well). Clearly, $\xi^0_{sk}=\xi^0(\mB^n)$ and $\xi^0_{sph}=\xi^0(\mS^n)$.

The above concepts assume the so-called $2$-spin interactions. Extension to $p$-spin ($p\geq 2$) gives
\begin{eqnarray}\label{eq:inteq1b}
\xi^0(p;\cS) = \lim_{n\rightarrow\infty} \frac{1}{\sqrt{n}}\mE_A \max_{\x\in\cS} \sum_{i_1,i_2,\dots,i_p=1}^n A_{i_1,i_2,\dots,i_p}\prod_{j=1}^{p}\x_{i_j},
\end{eqnarray}
with $A_{i_1,i_2,\dots,i_p}$  being independent standard normals and $\x_{i_j}$ being the $i_j$-th component of $\x$. Clearly, by definition, $\xi^0(\cS)\triangleq \xi^0(2;\cS)$. Our interest is in the so-called multipartite $p$-spin models where $j$-th interacting spins belong to (potentially different) sets  $\cS^{(j)},1\leq j\leq p$. The corresponding ground state free energy is then given as
\begin{eqnarray}\label{eq:inteq1c}
\xi(p;\bcS) = \lim_{n\rightarrow\infty} \frac{1}{\sqrt{n}}\mE_A \max_{\x^{(j)}\in\cS^{(j)}} \sum_{i_1,i_2,\dots,i_p=1}^n A_{i_1,i_2,\dots,i_p}\prod_{j=1}^{p}\x_{i_j}^{(j)},
\end{eqnarray}
where $\bcS=\lp \cS^{(j)}\rp_{j=1,\dots,p}$ is a $p$-dimensional array consisting of sets $\cS^{(j)}$.

The following theorem provides a generic \emph{non-asymptotic} strategy to upper/lower bound $\xi(p;\bcS) $.

\begin{theorem}
\label{thm:thm1}
  Consider  $n\in\mN$ and  $p\in 2\mN$. Let $A_{i_1,i_2,\dots,i_p}\sim \cN(0,1)$ for any choice of integers $i_j\in\{1,2,\dots,n\},1\leq j \leq p$. Also, let $g$ be a standard normal variable and for $j\in\{1,2,\dots,p\}$ let $\g^{(j)}\in \mR^{n\times 1}$ and $A^{(j)}\in \mR^{n\times n}$ be comprised of standard normals. Assume that all random variables are independent among themselves. Moreover, let $\cS^{(j)}\subseteq\mS^n, 1\leq j\leq p$ and for real scalars $c_3>0$, $c_{3,u}>0$, and $c_{3,l}>0$  set $\bar{\g}\triangleq\lp \g^{(j)}\rp_{j=1,\dots,p}$, $\bar{A}\triangleq\lp A^{(j)}\rp_{j=1,\dots,p}$, $\bcS\triangleq\lp \cS^{(j)}\rp_{j=1,\dots,p}$, and
  \begin{eqnarray}
   \label{eq:thm1eq1}
\varphi(p;\bcS,n) & = &  \max_{\x^{(j)}\in\cS^{(j)}} \sum_{i_1,i_2,\dots,i_p=1}^n A_{i_1,i_2,\dots,i_p}\prod_{j=1}^{p}\x_{i_j}^{(j)} \nonumber \\
\xi(p;\bcS,n,c_3) & = &  \frac{1}{\sqrt{n}} \frac{1}{c_3}\log \lp  \mE_A e^{c_3 \varphi(p;\bcS,n) } \rp \nonumber \\
\xi_u(p;\bcS,n,c_3) & = & \frac{1}{\sqrt{n}} \lp -\frac{c_{3}}{2}(p-1)  + \frac{1}{c_{3}} \log\lp  \mE_{\bar{\g}} e^{c_{3}\max_{\x^{(j)}\in\cS^{(j)}} \sum_{j=1}^p \lp\g^{(j)}\rp^T\x^{(j)} } \rp \rp
\nonumber \\
\xi_l(p;\bcS,n,c_3) & = &   \frac{1}{\sqrt{ n}} \frac{1}{c_{3}} \log \lp \mE_{\bar{A}} e^{ \sqrt{p}^{-1}c_{3} \max_{\x^{(j)}\in\cS^{(j)}} \sum_{j=1}^p \sum_{i_1,i_2,\dots,i_p=1}^n A^{(j)}_{i_1,i_2,\dots,i_p}\prod_{k=1}^{p}\x_{i_k}^{(j)}
 } \rp.
 \end{eqnarray}
Then
\begin{eqnarray}
   \label{eq:thm1eq2}
\xi_l(p;\bcS,n,c_{3,l})  \leq \xi(p;\bcS,n,c_{3,l}) \quad \mbox{and} \quad  \xi(p;\bcS,n,c_{3,u})   \leq  \xi_u(p;\bcS,n,c_{3,u}).
\end{eqnarray}
\end{theorem}

\begin{proof}
  We split the proof into two parts. The first part relates to the upper bound and the second to the lower bound.

\noindent  \underline{\emph{Upper bound}}: We define two centered Gaussian processes indexed by array $\bar{\x}=\lp \x^{(j)}\rp_{j=1,\dots,p}$
  \begin{eqnarray}
\label{eq:mr1}
\cG (\bar{\x}) & \triangleq & \cG (\x^{(1)},\x^{(2)},\dots,\x^{(p)})  \triangleq \sum_{i_1,i_2,\dots,i_p=1}^n A_{i_1,i_2,\dots,i_p}\prod_{j=1}^{p}\x_{i_j}^{(j)}   + g\sqrt{p-1}  \nonumber   \\
\cG_u (\bar{\x}) & \triangleq & \cG_u (\x^{(1)},\x^{(2)},\dots,\x^{(p)})  \triangleq \sum_{j=1}^p \lp\g^{(j)}\rp^T\x^{(j)} .
  \end{eqnarray}
Taking two arrays $\bar{\x}^{(a_1)}=\lp \x^{(j,a_1)}\rp_{j=1,\dots,p}$ and $\bar{\x}^{(a_2)}=\lp \x^{(j,a_2)}\rp_{j=1,\dots,p}$ we then write
  \begin{eqnarray}
\label{eq:mr2}
\mE \cG (\bar{\x}^{(a_1)})\cG (\bar{\x}^{(a_2)} )   & =  & \prod_{j=1}^{p}  \sum_{i_j=1}^n   \x_{i_j}^{(j,a_1)}  \x_{i_j}^{(j,a_2)}  + p-1
=  \prod_{j=1}^{p}  \lp \x^{(j,a_1)} \rp^T\x^{(j,a_2)}  + p-1 \nonumber   \\
\mE \cG_u (\bar{\x}^{(a_1)})\cG_u (\bar{\x}^{(a_2)} )   & =  & \sum_{j=1}^p \sum_{i_j=1}^n \x_{i_j}^{(j,a_1)}\x_{i_j}^{(j,a_2)}
=
\sum_{j=1}^p
 \lp \x^{(j,a_1)} \rp^T\x^{(j,a_2)}.
  \end{eqnarray}
  Clearly,
 \begin{eqnarray}
\label{eq:mr2a}
\mE \cG (\bar{\x}^{(a_1)})\cG (\bar{\x}^{(a_1)} )
=  \prod_{j=1}^{p}  \lp \x^{(j,a_1)} \rp^T\x^{(j,a_1)}  + p-1
= p =\sum_{j=1}^p
 \lp \x^{(j,a_1)} \rp^T\x^{(j,a_1)} =
\mE \cG_u (\bar{\x}^{(a_1)})\cG_u (\bar{\x}^{(a_1)} )  .
  \end{eqnarray}
On the other hand,   after setting $a^{(j)}\triangleq\lp \x^{(j,a_1)} \rp^T\x^{(j,a_2)} $ and noting that $|a^{(j)}|\leq 1$ we have for any $k\in\{1,2,\dots, p\}$
 \begin{eqnarray}
\label{eq:mr3}
  \mE \cG (\bar{\x}^{(a_1)})\cG (\bar{\x}^{(a_2)} )
 -
\mE \cG_u (\bar{\x}^{(a_1)})\cG_u (\bar{\x}^{(a_2)} )
 & =  &
  \prod_{j=1}^{p} a^{(j)}  + p-1
-\sum_{j=1}^p
 a^{(j)}
   \nonumber \\
& \geq &  - \prod_{j=1}^{k} a^{(j)} -\sum_{j=k+1}^p
 a^{(j)}
 + p-k  +  \prod_{j=1}^{p} a^{(j)}.
  \end{eqnarray}
To see this, we first observe that for $k=1$ (\ref{eq:mr3}) holds with equality. Assuming further that it holds for a fixed $k<p-1$ we then write
 \begin{eqnarray}
\label{eq:mr4}
  \mE \cG (\bar{\x}^{(a_1)})\cG (\bar{\x}^{(a_2)} )
 -
\mE \cG_u (\bar{\x}^{(a_1)})\cG_u (\bar{\x}^{(a_2)} )
 & \geq &  - \prod_{j=1}^{k} a^{(j)} -\sum_{j=k+1}^p
 a^{(j)}
 + p-k  +  \prod_{j=1}^{p} a^{(j)}
 \nonumber \\
 & = &  \lp 1 - a^{(k+1)}\rp \lp 1 - \prod_{j=1}^{k} a^{(j)} \rp
 \nonumber \\
& &  - \prod_{j=1}^{k+1} a^{(j)} -\sum_{j=k+2}^p
 a^{(j)}
 + p-k-1  +  \prod_{j=1}^{p} a^{(j)}
 \nonumber \\
& \geq &  - \prod_{j=1}^{k+1} a^{(j)} -\sum_{j=(k+1)+1}^p a^{(j)}
 + p-(k+1)  +  \prod_{j=1}^{p} a^{(j)}.
  \end{eqnarray}
The right hand side of the last inequality  is (\ref{eq:mr3})  for $k+1$ which then by induction implies that (\ref{eq:mr3}) indeed holds for $1\leq k\leq p-1$. Taking $k=p-1$ further gives
 \begin{eqnarray}
\label{eq:mr5}
  \mE \cG (\bar{\x}^{(a_1)})\cG (\bar{\x}^{(a_2)} )
 -
\mE \cG_u (\bar{\x}^{(a_1)})\cG_u (\bar{\x}^{(a_2)} )
 & \geq &  - \prod_{j=1}^{p-1} a^{(j)} -\sum_{j=p}^p
 a^{(j)}
 + 1  +  \prod_{j=1}^{p} a^{(j)}
\nonumber
\\
& = &
 \lp 1-a^{(j)}\rp \lp 1 - \prod_{j=1}^{p-1} a^{(j)} \rp
  \geq   0.
  \end{eqnarray}
Combining (\ref{eq:mr2a}) and (\ref{eq:mr5}) with results of \cite{Gordon85} (which are a special case of concepts discussed in Corollary 3 in \cite{Stojnicgscompyx16}  and in Corollary 4 in \cite{Stojnicgscomp16}) we obtain
\begin{eqnarray}
\label{eq:mr6}
  \mE \max_{\bar{\x}} e^{c_{3,u}\cG(\bar{\x})} &\leq &   \mE \max_{\bar{\x}} e^{c_{3,u}\cG_u(\bar{\x})},
\end{eqnarray}
where $\max_{\bar{\x}}$ means $\max_{\x^{(j)}\in \cS^{(j)}}$. A further combination of (\ref{eq:mr1}) and (\ref{eq:mr6}) gives
\begin{eqnarray}
\label{eq:mr7}
  \mE \max_{\bar{\x}} e^{c_{3,u}\lp \sum_{i_1,i_2,\dots,i_p=1}^n A_{i_1,i_2,\dots,i_p}\prod_{j=1}^{p}\x_{i_j}^{(j)}   + g\sqrt{p-1}   \rp  } &\leq &   \mE \max_{\bar{\x}} e^{c_{3,u} \sum_{j=1}^p \lp\g^{(j)}\rp^T\x^{(j)}   },
  \end{eqnarray}
which is equivalent to
\begin{eqnarray}
\label{eq:mr8}
   \mE e^{c_{3,u} \max_{\bar{\x}}  \lp \sum_{i_1,i_2,\dots,i_p=1}^n A_{i_1,i_2,\dots,i_p}\prod_{j=1}^{p}\x_{i_j}^{(j)}    \rp  } &\leq &  e^{-\frac{c_{3,u}^2}{2}(p-1) } \mE e^{c_{3,u}  \max_{\bar{\x}}  \sum_{j=1}^p \lp\g^{(j)}\rp^T\x^{(j)}   } ,
\end{eqnarray}
and
\begin{eqnarray}
\label{eq:mr9}
   \log \lp \mE e^{c_{3,u} \max_{\bar{\x}}  \lp \sum_{i_1,i_2,\dots,i_p=1}^n A_{i_1,i_2,\dots,i_p}\prod_{j=1}^{p}\x_{i_j}^{(j)}    \rp  }
   \rp
   &\leq &
   -\frac{c_{3,u}^2}{2}(p-1) + \log \lp
    \mE e^{c_{3,u}  \max_{\bar{\x}}  \sum_{j=1}^p \lp\g^{(j)}\rp^T\x^{(j)}   }
    \rp, \nonumber \\
\end{eqnarray}
which (after choosing $c_3=c_{3,u}$) establishes the second inequality in (\ref{eq:thm1eq2}).

\noindent \underline{\emph{Lower bound}}: For the lower bound portion we  consider the following two centered Gaussian processes
  \begin{eqnarray}
\label{eq:mr10}
\cG_x (\bar{\x}) & \triangleq &  \cG_x (\x^{(1)},\x^{(2)},\dots,\x^{(p)})  \triangleq \sqrt{p} \sum_{i_1,i_2,\dots,i_p=1}^n A_{i_1,i_2,\dots,i_p}\prod_{j=1}^{p}\x_{i_j}^{(j)}   \nonumber   \\
\cG_l (\bar{\x}) & \triangleq & \cG_u (\x^{(1)},\x^{(2)},\dots,\x^{(p)})  \triangleq \sum_{j=1}^p \sum_{i_1,i_2,\dots,i_p=1}^n A^{(j)}_{i_1,i_2,\dots,i_p}\prod_{k=1}^{p}\x_{i_k}^{(j)} .
  \end{eqnarray}
Taking again two arrays $\bar{\x}^{(a_1)}=\lp \x^{(j,a_1)}\rp_{j=1,\dots,p}$ and $\bar{\x}^{(a_2)}=\lp \x^{(j,a_2)}\rp_{j=1,\dots,p}$ we write
  \begin{eqnarray}
\label{eq:mr12}
\mE \cG_x (\bar{\x}^{(a_1)})\cG_x (\bar{\x}^{(a_2)} )   & =  &  p \prod_{j=1}^{p}  \sum_{i_j=1}^n   \x_{i_j}^{(j,a_1)}  \x_{i_j}^{(j,a_2)}
=  p \prod_{j=1}^{p}  \lp \x^{(j,a_1)} \rp^T\x^{(j,a_2)} =  p \prod_{j=1}^{p}  a^{(j)} \nonumber   \\
\mE \cG_l (\bar{\x}^{(a_1)})\cG_l (\bar{\x}^{(a_2)} )   & =  & \sum_{j=1}^p  \lp \sum_{i_j=1}^n \x_{i_j}^{(j,a_1)}\x_{i_j}^{(j,a_2)} \rp^p
=
\sum_{j=1}^p \lp a^{(j)}\rp^p.
  \end{eqnarray}
  Clearly,
 \begin{eqnarray}
\label{eq:mr12a}
\mE \cG_x (\bar{\x}^{(a_1)})\cG_x (\bar{\x}^{(a_1)} )
 = p  =
\mE \cG_l (\bar{\x}^{(a_1)})\cG_l (\bar{\x}^{(a_1)} )  .
  \end{eqnarray}
Moreover, arithmetic-geometric mean inequality gives
 \begin{eqnarray}
\label{eq:mr13}
 \mE \cG_l (\bar{\x}^{(a_1)})\cG_l (\bar{\x}^{(a_2)} )  =
 \sum_{j=1}^p \lp a^{(j)}\rp^p  = p \frac{ \sum_{j=1}^p \lp a^{(j)}\rp^p }{p}
  \geq  p \lp \prod_{j=1}^p \lp a^{(j) } \rp^p \rp^{\frac{1}{p}} =\mE \cG_x (\bar{\x}^{(a_1)})\cG_x (\bar{\x}^{(a_2)} )  .
  \end{eqnarray}
Combining (\ref{eq:mr12a}) and (\ref{eq:mr13}) with results of \cite{Gordon85}  we obtain
\begin{eqnarray}
\label{eq:mr14}
  \mE  \max_{\bar{\x}} e^{c_{3,l}\cG_l(\bar{\x})} &\leq &   \mE  \max_{\bar{\x}} e^{c_{3,l}\cG_x(\bar{\x})}.
\end{eqnarray}
A further combination of (\ref{eq:mr10}) and (\ref{eq:mr14}) gives
\begin{eqnarray}
\label{eq:mr15}
  \mE \max_{\bar{\x}} e^{c_{3,l}   \sum_{j=1}^p \sum_{i_1,i_2,\dots,i_p=1}^n A^{(j)}_{i_1,i_2,\dots,i_p}\prod_{k=1}^{p}\x_{i_k}^{(j)}   }
  &\leq &
  \mE \max_{\bar{\x}} e^{\sqrt{p} c_{3,l}\lp \sum_{i_1,i_2,\dots,i_p=1}^n A_{i_1,i_2,\dots,i_p}\prod_{j=1}^{p}\x_{i_j}^{(j)}    \rp  } ,
  \end{eqnarray}
which is equivalent to
\begin{equation}
\label{eq:mr16}
\log \lp   \mE  e^{c_{3,l}  \max_{\bar{\x}}  \sum_{j=1}^p \sum_{i_1,i_2,\dots,i_p=1}^n A^{(j)}_{i_1,i_2,\dots,i_p}\prod_{k=1}^{p}\x_{i_k}^{(j)}   } \rp
  \leq
\log \lp  \mE e^{\sqrt{p} c_{3,l}  \max_{\bar{\x}}  \sum_{i_1,i_2,\dots,i_p=1}^n A_{i_1,i_2,\dots,i_p}\prod_{j=1}^{p}\x_{i_j}^{(j)}    }  \rp.
\end{equation}
After scaling both sides by $\sqrt{p}c_{3,l}$ and choosing $c_{3}=\sqrt{p}c_{3,l}$, (\ref{eq:mr16})  is sufficient to show the first inequality in (\ref{eq:thm1eq2}) and complete the whole proof.
 \end{proof}

\begin{remark}
\label{rem:rem1}
  It is useful to note that the second inequality in (\ref{eq:thm1eq2}) actually holds for $p\in\mN$.
\end{remark}

\begin{remark}
\label{rem:rem2}
For the elegance of the exposition we have chosen $\x^{(j)}\in\mR^n$. It is trivial to rewrite everything with $\x^{(j)}\in\mR^{n_j}$ (i.e., with $\cS^{(j)}\subseteq \mS^{n_j}$). In other words, everything stated above immediately extends to spins belonging to sets of unequal dimensions (in the multi-species terminology this means that species of different length are easily incorporated as well).
\end{remark}

The following corollary can be established as well.
\begin{corollary}
\label{cor:cor1}
 Assume the setup of Theorem \ref{thm:thm1}. Recall
\begin{eqnarray}
   \label{eq:cor1eq1}
\xi_u(p;\bcS,n,c_3) & = & \frac{1}{\sqrt{n}} \lp -\frac{c_{3}}{2}(p-1)  + \frac{1}{c_{3}} \log\lp  \mE_{\bar{\g}} e^{c_{3}\max_{\x^{(j)}\in\cS^{(j)}} \sum_{j=1}^p \lp\g^{(j)}\rp^T\x^{(j)} } \rp \rp
\nonumber \\
& = & \frac{1}{\sqrt{n}} \lp -\frac{c_{3}}{2}(p-1)  + \frac{1}{c_{3}} \sum_{j=1}^p  \log\lp  \mE_{\bar{\g}} e^{c_{3}\max_{\x^{(j)}\in\cS^{(j)}} \lp\g^{(j)}\rp^T\x^{(j)} } \rp \rp,
\end{eqnarray}
and set
\begin{eqnarray}
   \label{eq:cor1eq2}
 \xi_l(p;\bcS,n,0) \triangleq \lim_{c_{3}\rightarrow 0}\xi_l(p;\bcS,n,c_{3}) & = &  \frac{1}{\sqrt{p n}}\lp \mE_{\bar{A}} \max_{\x^{(j)}\in\cS^{(j)}} \sum_{j=1}^p \sum_{i_1,i_2,\dots,i_p=1}^n A^{(j)}_{i_1,i_2,\dots,i_p}\prod_{k=1}^{p}\x_{i_k}^{(j)}
 \rp.
\end{eqnarray}
 In the thermodynamic limit
\begin{eqnarray}
   \label{eq:cor1eq3}
\xi(p;\bcS) = \lim_{n\rightarrow\infty} \frac{1}{\sqrt{n}}\mE_A \varphi(p;\bcS,n)
= \lim_{n\rightarrow\infty} \frac{1}{\sqrt{n}}\mE_A \max_{\x^{(j)}\in\cS^{(j)}} \sum_{i_1,i_2,\dots,i_p=1}^n A_{i_1,i_2,\dots,i_p}\prod_{j=1}^{p}\x_{i_j}^{(j)},
\end{eqnarray}
and
\begin{eqnarray}
   \label{eq:cor1eq4}
\lim_{n\rightarrow \infty} \xi_l(p;\bcS,n,0)  \leq \xi(p;\bcS)  \leq \min_{c_{3}>0} \lim_{n\rightarrow \infty}  \xi_u(p;\bcS,n,c_{3}).
\end{eqnarray}
Moreover, for $\cS^{(j)}=\cS,1\leq j\leq p$, set $\g=\g^{(1)}$ and
\begin{eqnarray}
   \label{eq:cor1eq5}
  \xi_u^0(p;\cS,n,c_3) & \triangleq &  \frac{ \sqrt{p}}{\sqrt{n}} \lp -\frac{c_{3}}{2}(p-1)  + \frac{1}{c_{3}} \log\lp  \mE_{\g} e^{ c_{3} \sqrt{p} \max_{\x\in\cS } \g^T\x } \rp \rp  \nonumber \\
  \xi_l^0(p;\cS,n) & \triangleq & \frac{ \sqrt{p}}{\sqrt{ n}}\lp \mE_{A} \max_{\x\in\cS} \sum_{i_1,i_2,\dots,i_p=1}^n A_{i_1,i_2,\dots,i_p}\prod_{k=1}^{p}\x_{i_k}
 \rp.
\end{eqnarray}
Then
\begin{eqnarray}
   \label{eq:cor1eq6}
\lim_{n\rightarrow \infty} \xi_l^0(p;\cS,n)  \leq \xi(p;\cS)
 \leq \min_{c_{3}>0} \lim_{n\rightarrow \infty}  \xi_u^0(p;\cS,n,c_{3}),
\end{eqnarray}
where $ \xi(p;\cS)\triangleq\xi(p;\bcS)$ when $\cS^{(j)}=\cS,1\leq j\leq p$.
 \end{corollary}

\begin{proof}
The concavity of $\log(\cdot)$ implies
\begin{eqnarray}
\label{eq:mr17}
   c_{3,u} \mE  \max_{\bar{\x}}  \lp \sum_{i_1,i_2,\dots,i_p=1}^n A_{i_1,i_2,\dots,i_p}\prod_{j=1}^{p}\x_{i_j}^{(j)} \rp
   &\leq &
   \log \lp \mE e^{c_{3,u} \max_{\bar{\x}}  \lp \sum_{i_1,i_2,\dots,i_p=1}^n A_{i_1,i_2,\dots,i_p}\prod_{j=1}^{p}\x_{i_j}^{(j)}    \rp  }
   \rp, \nonumber \\
\end{eqnarray}
which together with (\ref{eq:mr8}) establishes the second inequality in (\ref{eq:cor1eq4}) and completes the upper bound portion of the corollary. On the other hand, the first inequality in (\ref{eq:cor1eq4}) follows automatically from the first inequality in (\ref{eq:thm1eq2}) and  (\ref{eq:cor1eq2}) after  additionally noting
\begin{eqnarray}
   \label{eq:mr18}
 \lim_{n\rightarrow \infty} \xi(p;\bcS,n,0)
 & = &
 \lim_{c_{3}\rightarrow 0}
 \xi(p;\bcS,n,c_3)  =   \lim_{n\rightarrow \infty}  \lim_{c_{3}\rightarrow 0} \frac{1}{\sqrt{n}} \frac{1}{c_3}\log \lp  \mE_A e^{c_3 \varphi(p;\bcS,n) } \rp
\nonumber \\
& = &  \lim_{n\rightarrow \infty}  \frac{1}{\sqrt{n}}   \mE_A  \varphi(p;\bcS,n) =  \xi(p;\bcS).
\end{eqnarray}
Finally, (\ref{eq:cor1eq6}) follows automatically from (\ref{eq:cor1eq4}) after a cosmetic change of variables, $c_3\rightarrow c_3\sqrt{p}$, in the second inequality.
\end{proof}

The above theorem and corollary provide bounds as functions of single partite systems. The second part of the corollary relates to the same type of single partite systems where the restricting unit sphere subsets are actually identical.   A large deviation principle (LDP) associated with such systems is discussed below and its several remarkable properties are uncovered.

\section{Large deviations}
\label{sec:ldp}

We consider function
 \begin{eqnarray}
 \label{eq:ldpeq1}
   \zeta (p;\cS,n) & \triangleq & \frac{1}{\sqrt{ n}}\lp \max_{\x\in\cS} \sum_{i_1,i_2,\dots,i_p=1}^n A_{i_1,i_2,\dots,i_p}\prod_{k=1}^{p}\x_{i_k}
 \rp,
  \end{eqnarray}
which is the  underlying randomness of $\xi_l^0(p;\cS,n) $ in Corollary \ref{cor:cor1}. Before discussing its associated LDP, we find it convenient to establish the following corollary of Theorem \ref{thm:thm1}.

\begin{corollary}
\label{cor:cor2}
Assume the setup of Theorem \ref{thm:thm1} with $\g=\g^{(j)}$, $\cS^{(j)}=\cS,1\leq j\leq p$, and
 \begin{eqnarray}
   \label{eq:cor2eq1}
 \xi_u^0(p;\cS,n,c_3) & = & \frac{\sqrt{p}}{\sqrt{n}} \lp -\frac{c_{3}}{2}(p-1)  + \frac{1}{c_{3}} \log\lp  \mE_{\g} e^{c_{3}\sqrt{p}\max_{\x\in\cS} \g^T\x } \rp \rp
\nonumber \\
\xi_l^0(p;\cS,n,c_3) & = &  \frac{\sqrt{p}}{\sqrt{n}} \frac{1}{c_{3}} \log \lp \mE_{A} e^{ c_{3}  \sqrt{n}     \zeta (p;\cS,n)  } \rp  =  \frac{\sqrt{p}}{\sqrt{n}} \frac{1}{c_{3}} \log \lp \mE_{A} e^{ c_{3} \max_{\x\in\cS }  \sum_{i_1,i_2,\dots,i_p=1}^n A_{i_1,i_2,\dots,i_p}\prod_{k=1}^{p}\x_{i_k}
 } \rp.\nonumber \\
 \end{eqnarray}
Then
\begin{eqnarray}
   \label{eq:cor2eq2}
\xi^0_l(p;\cS,n,c_{3})  \leq  \xi^0_u(p;\cS,n,c_{3}).
\end{eqnarray}
\end{corollary}

\begin{proof}
For $p\in2\mN$ the above statement follows automatically by connecting the two inequalities in (\ref{eq:thm1eq2}) with a cosmetic change $c_{3} = c_{3}\sqrt{p}$. It is useful to note that it can also be obtained separately for $p\in\mN$ by adapting the reasoning of Theorem \ref{thm:thm1}. Namely, one defines two centered Gaussian processes indexed by $\x\in\cS \subseteq \mS^n$
  \begin{eqnarray}
\label{eq:cor2mr1}
\cG^0 (\x) & = & \sum_{i_1,i_2,\dots,i_p=1}^n A_{i_1,i_2,\dots,i_p}\prod_{j=1}^{p}\x_{i_j}   + g\sqrt{p-1}  \nonumber   \\
\cG_u^0 (\x) & = &   \sqrt{p} \g^T\x .
  \end{eqnarray}
For $\x^{(a_1)}$ and $\x^{(a_2)} $ the following analogue of (\ref{eq:mr2}) can be written
  \begin{eqnarray}
\label{eq:cor2mr2}
\mE \cG^0 (\x^{(a_1)})\cG^0 (\x^{(a_2)} )   & =  & \lp  \sum_{i_j=1}^n   \x_{i_j}^{(a_1)}  \x_{i_j}^{(a_2)}  \rp^p + p-1
=  \lp \lp \x^{(a_1)} \rp^T\x^{(a_2)} \rp^p + p-1 \nonumber   \\
\mE \cG^0_u (\x^{(a_1)})\cG^0_u (\x^{(a_2)} )   & =  & p \sum_{i_j=1}^n \x_{i_j}^{(a_1)}\x_{i_j}^{(a_2)}
=
p
 \lp \x^{(a_1)} \rp^T\x^{(a_2)}.
  \end{eqnarray}
  One then notes
 \begin{eqnarray}
\label{eq:mr2a}
\mE \cG^0 (\x^{(a_1)})\cG^0 (\x^{(a_1)} )
=  \lp \lp \x^{(a_1)} \rp^T\x^{(a_1)} \rp^p  + p-1
= p = p
 \lp \x^{(a_1)} \rp^T\x^{(a_1)} =
\mE \cG^0_u (\x^{(a_1)})\cG^0_u (\x^{(a_1)} ),
  \end{eqnarray}
and, after setting $a^{0}\triangleq\lp \x^{(a_1)} \rp^T\x^{(a_2)} $, repeats line-by-line all the steps between (\ref{eq:mr3}) and  (\ref{eq:mr6}) with $a^{(j)}=a^{(0)},1\leq j\leq p$, to arrive at the following
\begin{eqnarray}
\label{eq:cor2mr6}
  \mE \max_{\x} e^{c_{3}\cG^0(\x)} &\leq &   \mE \max_{\x} e^{c_{3}\cG^0_u(\x)}.
\end{eqnarray}
A combination of (\ref{eq:cor2mr1}) and (\ref{eq:cor2mr6}) gives
\begin{eqnarray}
\label{eq:cor2mr7}
  \mE \max_{\bar{\x}} e^{c_{3}\lp \sum_{i_1,i_2,\dots,i_p=1}^n A_{i_1,i_2,\dots,i_p}\prod_{k=1}^{p}\x_{i_k}   + g\sqrt{p-1}   \rp  } &\leq &   \mE \max_{\bar{\x}} e^{c_{3} \sqrt{p} \g^T\x   },
  \end{eqnarray}
which is equivalent to
 \begin{eqnarray}
\label{eq:cor2mr9}
   \log \lp \mE e^{c_{3} \max_{\x}  \lp \sum_{i_1,i_2,\dots,i_p=1}^n A_{i_1,i_2,\dots,i_p}\prod_{k=1}^{p}\x_{i_k}    \rp  }
   \rp
   &\leq &
   -\frac{c_{3}^2}{2}(p-1) + \log \lp
    \mE e^{c_{3} \sqrt{p}  \max_{\x}  \g^T\x   }
    \rp,
\end{eqnarray}
and, after scaling by $c_3$, sufficient to complete the proof.

 \end{proof}

We below discuss two specializations of predominant interest in mathematics of spin glasses literature: (i) the spherical and (ii) the Ising spins.

\subsection{Spherical specialization}
\label{sec:ldpsph}

\begin{theorem}
  \label{thm:thm3}
 Assume the setup of Theorem \ref{thm:thm1} and Corollary \ref{cor:cor1}  with $\cS^{(j)} = \cS = \mS^n$. Set
 \begin{eqnarray}
 \label{eq:thm3eq0}
u_* & \triangleq  & 2 \sqrt{\frac{ p-1 }{p} }
\nonumber \\
\phi_{\mS^n}(p,u)
& \triangleq &
  \frac{1}{2} \log \lp p-1  \rp  -\frac{p-2}{4(p-1)}u^2
  - \frac{ \sqrt{ u^2 - u_*^2   }   }{u_*^2 } u  + \log\lp \frac{ u + \sqrt{u^2 -  u_*^2 } }{  u_*    }  \rp ,  u\geq u_*.
 \end{eqnarray}
 For $\zeta(\cdot)$ from  (\ref{eq:ldpeq1}) one then has
 \begin{eqnarray}
 \label{eq:thm3eq1}
 \lim_{n\rightarrow \infty }  \frac{1}{n} \log \lp \mP \lp \zeta(p;\mS^n,n) \geq u \rp \rp  \leq \phi_{\mS^n}(p,u) .
  \end{eqnarray}
\end{theorem}

\begin{proof}
  By Chernoff/Markov inequality we have
 \begin{eqnarray}
 \label{eq:ldpeq2}
\frac{1}{n}\log \lp  \mP \lp \zeta(p;\mS^n,n) \geq u \rp \rp
 \leq  \frac{1}{n} \log \lp e^{c_3 \sqrt{n} \zeta(p;\mS^n,n) -c_3\sqrt{n} u} \rp
 =  \frac{1}{n} \log \lp e^{c_3 \sqrt{n} \zeta(p;\mS^n,n) } \rp -\frac{1}{n}c_3\sqrt{n} u .
  \end{eqnarray}
Combining (\ref{eq:ldpeq2})  with Corollary \ref{cor:cor2} and adopting scaling $c_3\rightarrow c_3\sqrt{n}$ we find
 \begin{eqnarray}
 \label{eq:ldpeq3}
\frac{1}{n}\log \lp  \mP \lp \zeta(p;\mS^n,n) \geq u \rp \rp
& \leq &
 -\frac{c_{3}^2}{2}(p-1)  +  \frac{1}{n}  \log\lp  \mE_{\g} e^{c_{3}\sqrt{n}\sqrt{p}\max_{\x\in\mS^n} \g^T\x } \rp
-c_3 u
\nonumber \\
& = &
  -\frac{c_{3}^2}{2}(p-1)  +  \frac{1}{n}  \log\lp  \mE_{\g} e^{c_{3}\sqrt{n}\sqrt{p} \| \g\|_2} \rp
-c_3 u.
  \end{eqnarray}
Utilizing results of \cite{Stojnicl1RegPosasym,StojnicMoreSophHopBnds10} we have
 \begin{eqnarray}
 \label{eq:ldpeq4}
 \lim_{n\rightarrow \infty} \frac{1}{n}  \log\lp  \mE_{\g} e^{c_{3}\sqrt{n}\sqrt{p} \| \g\|_2} \rp
=
\min_{\gamma > 0 } \phi_1(\gamma),
\end{eqnarray}
with
 \begin{eqnarray}
 \label{eq:ldpeq4a}
\phi_1(\gamma) = \gamma c_3\sqrt{p} -\frac{1}{2} \log\lp 1- \frac{c_3\sqrt{p}}{2\gamma}\rp.
  \end{eqnarray}
  Computing derivative and equaling it to zero one first obtains
 \begin{eqnarray}
 \label{eq:ldpeq4a}
\frac{d\phi_1(\gamma)}{d\gamma} = c_3\sqrt{p} - \frac{1}{ 2\gamma - c_3\sqrt{p}}  + \frac{1}{2\gamma}
=c_3\sqrt{p} \lp 1 - \frac{1}{ 2\gamma (2\gamma - c_3\sqrt{p} ) } \rp = 0,
  \end{eqnarray}
and then finds the optimal $\gamma$
 \begin{eqnarray}
 \label{eq:ldpeq5}
\hat{ \gamma } = \frac{c_3\sqrt{p}  +  \sqrt{c_3^2 p +4} }{4}.
  \end{eqnarray}
A combination of  (\ref{eq:ldpeq3})-(\ref{eq:ldpeq5}) gives
 \begin{eqnarray}
 \label{eq:ldpeq6}
\frac{1}{n}\log \lp  \mP \lp \zeta(p;\mS^n,n) \geq u \rp \rp
& \leq &
\phi_2(c_3),
\end{eqnarray}
with
 \begin{eqnarray}
 \label{eq:ldpeq7}
\phi_2(c_3)
=   -\frac{c_{3}^2}{2}(p-1)  +  \hat{ \gamma } c_3\sqrt{p} -\frac{1}{2} \log\lp 1- \frac{c_3\sqrt{p}}{2\hat{\gamma}}\rp
-c_3 u.
  \end{eqnarray}
As (\ref{eq:ldpeq7})  holds for any $c_3$, one can further optimize the right hand side over $c_3$. Computing again derivative and equaling it to zero gives
 \begin{eqnarray}
 \label{eq:ldpeq8}
\frac{d\phi_2(c_3)}{d c_3}
 =   -c_{3}(p-1)  +   \frac{d\phi_1(\hat{\gamma})}{d c_3}+   \frac{d\phi_1(\hat{\gamma})}{d \hat{\gamma}}\frac{d\hat{\gamma}}{d c_3}
- u
 =   -c_{3}(p-1)  +  \hat{\gamma}\sqrt{p} + \frac{\sqrt{p}}{2(2\hat{\gamma} -c_3\sqrt{p}  )}     - u = 0.
  \end{eqnarray}
 Keeping in mind that $\hat{\gamma}$ satisfies (\ref{eq:ldpeq4a}), we further have
 \begin{eqnarray}
 \label{eq:ldpeq9}
\frac{d\phi_2(c_3)}{d c_3}
  =   -c_{3}(p-1)  + 2 \hat{\gamma}\sqrt{p}   - u
  =
 -c_{3}(p-1)  + \frac{c_3 p  +  \sqrt{c_3^2 p^2 +4p} }{2}   - u = 0.
  \end{eqnarray}
After additional algebraic rearranging one arrives at
 \begin{eqnarray}
 \label{eq:ldpeq10}
c_{3}^2(1-p) + u(p-2)c_3  + u^2-p   =0,
  \end{eqnarray}
 which allows to obtain
 \begin{eqnarray}
 \label{eq:ldpeq11}
\hat{c}_{3}   = \frac{ -(p-2)u \pm \sqrt{ ((p-2)u)^2 -4(1-p)(u^2-p)   }   }{2(1-p)}
=\frac{ -(p-2)u \pm \sqrt{ (pu)^2 - 4(p-1)p   }   }{2(1-p)}
.
  \end{eqnarray}
Combining (\ref{eq:ldpeq4a}) and (\ref{eq:ldpeq7}), we obtain
 \begin{eqnarray}
 \label{eq:ldpeq12}
\phi_2(\hat{c}_3)
=   -\frac{c_{3}^2}{2}(p-1)  +  \hat{ \gamma } c_3\sqrt{p} +\frac{1}{2} \log\lp 4 \hat{\gamma}^2\rp
-c_3 u.
  \end{eqnarray}
A further combination of (\ref{eq:ldpeq9}) and (\ref{eq:ldpeq12})  gives
 \begin{eqnarray}
 \label{eq:ldpeq13}
\phi_2(\hat{c}_3)
=  -\frac{1}{2}\hat{c}_3 u + \log\lp 2 \hat{\gamma}\rp.
  \end{eqnarray}
From (\ref{eq:ldpeq9}) and (\ref{eq:ldpeq13}) we have
 \begin{eqnarray}
 \label{eq:ldpeq14}
\phi_2(\hat{c}_3)
=  -\frac{1}{2}\hat{c}_3 u +  \frac{1}{2} \log\lp \lp \hat{c}_3(p-1)+u   \rp \rp  -  \frac{1}{4} \log (p) .
  \end{eqnarray}
After plugging $\hat{c}_3$ from  (\ref{eq:ldpeq11}) into (\ref{eq:ldpeq14}) one obtains
 \begin{eqnarray}
 \label{eq:ldpeq15}
\phi_2(\hat{c}_3)
& = &  -\frac{1}{2}  \frac{ -(p-2)u \pm \sqrt{ (pu)^2 - 4(p-1)p   }   }{2(1-p)} u + \log\lp \frac{ up \mp \sqrt{(up)^2 -4(p-1)p  } }{2\sqrt{p}}  \rp
\nonumber \\
& = &
  -\frac{p-2}{4(p-1)}u^2
  \pm \frac{ \sqrt{ u^2 - \frac{4(p-1)}{p}   }   }{4\frac{p-1}{p}} u  +  \log\lp  \frac {  u \mp \sqrt{  u^2 -\frac{4(p-1)}{p}  } }  { \sqrt{\frac{4}{p} }    }  \rp
\nonumber \\
 & = &
  -\frac{p-2}{4(p-1)}u^2
  \pm \frac{ \sqrt{ u^2 - \frac{4(p-1)}{p}   }   }{4\frac{p-1}{p}} u  + \log\lp \frac{ u \mp \sqrt{u^2 - \frac{4(p-1)}{p}  } }{\sqrt{\frac{4(p-1)}{p} } }  \rp  + \frac{1}{2} \log \lp p-1  \rp    .
  \end{eqnarray}
Recalling on the definition of $u_{*}$
\begin{eqnarray}
 \label{eq:ldpeq16}
u_* = 2 \sqrt{\frac{(p-1)}{p} } ,
  \end{eqnarray}
and choosing minus sign one arrives at
 \begin{eqnarray}
 \label{eq:ldpeq17}
\phi_2(\hat{c}_3)
  & = &
 \frac{1}{2} \log \lp p-1  \rp  -\frac{p-2}{4(p-1)}u^2
  - \frac{ \sqrt{ u^2 - u_*^2   }   }{u_*^2 } u  + \log\lp \frac{ u + \sqrt{u^2 -  u_*^2 } }{  u_*    }  \rp = \phi_{\mS^n}(p,u),
  \end{eqnarray}
  which is the theorem's claim.
\end{proof}

One now notes that the above theorem is very closely related to  \cite{Auff13} and consequently to \cite{Auff13a,Subag17,Subag17a} as well. The $n$-scaled logarithm of the expected number of the critical points of the pure $p$-spin spherical model was considered in  \cite{Auff13} (see also \cite{SubZeit17,SubZeit21} for related distributional and concentration considerations and \cite{McKenna24,Genovese22,AuffChen14,Kivimae23} for corresponding bipartite ones; relevant statistical physics considerations can be found in, e.g., \cite{Fyod04}).  The LDP good rate function $\phi(p,u)$ discussed in the above theorem precisely matches the critical points exponent obtained in \cite{Auff13}. Moreover, the value $u_*$ above which the presented formalism is particularly relevant precisely matches the limiting energy value obtained in \cite{Auff13} above which there is still a constant number of critical points. As in \cite{Auff13}, one can define
 \begin{eqnarray}
 \label{eq:ldpeq18}
u^{(sph)}_{GS}   \triangleq
\min & &  u \nonumber \\
\mbox{such that} & & \phi_{\mS^n}(p,u) <0,
  \end{eqnarray}
  or equivalently in our context (basically utilizing (\ref{eq:ldpeq7}))
 \begin{eqnarray}
 \label{eq:ldpeq18a}
u^{(sph)}_{GS}   \triangleq
\min_{c_3>0} \lp -\frac{c_{3}}{2}(p-1)  +  \hat{ \gamma } \sqrt{p} -\frac{1}{2c_3} \log\lp 1- \frac{c_3\sqrt{p}}{2\hat{\gamma}}\rp \rp.
    \end{eqnarray}
Then the ground state energy (GSE) of the spherical pure $p$-spin model is given as
 \begin{eqnarray}
 \label{eq:ldpeq18a}
\xi^0_{sph}(p) = u^{(sph)}_{GS},
   \end{eqnarray}
   and  the corresponding GSE of the spherical multipartite $p$-spin model as
 \begin{eqnarray}
 \label{eq:ldpeq19}
\xi_{sph}(p) = \sqrt{p}u^{(sph)}_{GS}.
   \end{eqnarray}

\vspace{.2in}

\noindent \underline{\emph{\textbf{Analytical agreement of results from \cite{Subagtap23b} and \cite{DarMc24}}}}

In \cite{Subagtap23b}, Subag obtain characterization of spherical multipartite $p$-spin models GSEs provided that ceratin energy mixtures converge. Using different methods, in \cite{DarMc24}, Dartois and McKenna obtained the upper bounds for the same GSEs and observed that they numerically match Subag's results. Our proving methods allow to explicitly establish full analytical agrement between these two sets of results.

One first notes that Lemma 2.7 in \cite{DarMc24} gives that the GSE upper bound, $E_0$, satisfies (we consider the nontrivial case $E_0\sqrt{\frac{p}{p-1}} \geq 2$)
\begin{equation}
\label{eq:agreq1}
 \frac{1+\log(p-1)}{2} + \frac{E^2_0\frac{p}{p-1}}{4} -\frac{1}{2}  - \frac{|E_0|\sqrt{\frac{p}{p-1}}}{4} \sqrt{E_0^2\frac{p}{p-1} - 4 }
 + \log \lp \sqrt{\frac{E_0^2\frac{p}{p-1}}{4} -1 }  + \frac{|E_0|\sqrt{\frac{p}{p-1}}}{2}  \rp   - \frac{E_0^2}{2}
   = 0.
 \end{equation}
The above is then equivalent to
\begin{equation}
\label{eq:agreq2}
 \frac{\log(p-1)}{2} - \frac{E^2_0\frac{p-2}{p-1}}{4}   - \frac{|E_0|}{4\frac{p-1}{p}} \sqrt{E_0^2 - 4\frac{p-1}{p} }
 + \log \lp \frac{\sqrt{E_0^2 - \frac{4(p-1)}{p} }  + |E_0| }{\sqrt{ \frac{4(p-1)}{p} } }  \rp
   = 0.
 \end{equation}
A combination of (\ref{eq:ldpeq17}), (\ref{eq:ldpeq18}), and (\ref{eq:agreq2}) together with a change of variables  $ E_0 \leftrightarrow u^{(sph)}_{(GS)}$ shows that GSE upper bound from \cite{DarMc24} indeed matches the exact value that we obtained above.

For Subag's part we have from Theorem 1 in \cite{Subagtap23b} that the ground state energy $E_*$ is obtained as
\begin{equation}
\label{eq:agreq3}
E_*    = \sqrt{ \lp -\log\lp 1 -q_s  \rp  \lp 1+p\lp\frac{1}{q_s}-1\rp \rp  \rp   },
 \end{equation}
with $q_s$ satisfying
\begin{equation}
\label{eq:agreq4}
 \frac{q_s^2}{p(1-q_s)}  =  \frac{-\log \lp 1 -q_s \rp }{ 1+p\lp\frac{1}{q_s}-1\rp  }.
 \end{equation}
Setting $z=\frac{1}{q_s}-1$ we then have
\begin{equation}
\label{eq:agreq5}
 \frac{1}{pz(1+z)}  =  \frac{E_*^2}{(1+pz)^2},
 \end{equation}
which is equivalent to
\begin{equation}
\label{eq:agreq6}
z^2(E^2_*p-p^2) + z (pE^2_* -2p) -1 =z^2(E^2_*-p)p + z (E^2_* -2)p -1 =0.
 \end{equation}
Solving over $z$ gives
\begin{equation}
\label{eq:agreq7}
\hat{z} = \frac{-(E^2_* -2)p  \pm \sqrt{((E^2_* -2)p )^2 +4 (E^2_*-p)p } }{2p(E^2_*-p)}
= \frac{-(E^2_* -2)  \pm  E_*  \sqrt{E^2_* - \frac{4(p-1)}{p}    }   }{2(E^2_*-p)},
 \end{equation}
 and consequently
\begin{equation}
\label{eq:agreq8}
\hat{q}_s=\frac{1}{1+\hat{z}}.
 \end{equation}
 From (\ref{eq:agreq3})  and (\ref{eq:agreq5}) one then obtains
\begin{equation}
\label{eq:agreq9}
  \frac{1+p\hat{z}}{p\hat{z} (1+\hat{z})} = \frac { E_*^2 } { \lp 1+p\hat{z} \rp } = -\log (1-\hat{q}_s) = -\log \lp  \frac{\hat{z}}{1+\hat{z}} \rp.
 \end{equation}
If we can show that for $u=u_{GS}^{(sph)}=E_*$ the following two equalities hold
\begin{eqnarray}
\label{eq:agreq10}
  \frac{1}{4\hat{\gamma}^2}  & = &  \frac{\hat{z}}{1+\hat{z}} \\
    -\hat{c}_3 u  & =  &  -\frac{1+p\hat{z}}{p\hat{z} (1+\hat{z})} ,\label{eq:agreq10a}
  \end{eqnarray}
then Subag's results are in agreement with our own and then automatically with \cite{DarMc24} as well.

\vspace{.2in}

\noindent\underline{\emph{Proof of (\ref{eq:agreq10}):}} From (\ref{eq:ldpeq11}) we first observe
 \begin{eqnarray}
 \label{eq:agreq11}
\frac{-2(p-1)\hat{c}_{3}-2u}{p}   =  -u - \sqrt{ u^2 - \frac{4(p-1)}{p}   }
=   \frac{\frac{4(p-1)}{p}  }{-u + \sqrt{ u^2 - \frac{4(p-1)}{p}   } } ,
  \end{eqnarray}
and then obtain
 \begin{eqnarray}
 \label{eq:agreq12}
\frac{1}{\frac{-2(p-1)\hat{c}_{3}-2u}{4(p-1)} }  =  -u + \sqrt{ u^2 - \frac{4(p-1)}{p}   } .
  \end{eqnarray}
Combining (\ref{eq:ldpeq9})  and (\ref{eq:agreq12}) we further have
 \begin{eqnarray}
 \label{eq:agreq12a}
\frac{p-1}{\sqrt{p}\hat{\gamma}}  =  -u + \sqrt{ u^2 - \frac{4(p-1)}{p}   } .
  \end{eqnarray}
Taking the plus sign in (\ref{eq:agreq7}) one finds
\begin{equation}
\label{eq:agreq13}
\frac{2(E^2_*-p) \hat{z} - 2}{E_*} =  -E_*   +  \sqrt{E^2_* - \frac{4(p-1)}{p}    }   ,
 \end{equation}
which for $E_*=u$ gives
\begin{equation}
\label{eq:agreq14}
\frac{2(u^2-p) \hat{z} - 2}{u} =  -u   +  \sqrt{u^2 - \frac{4(p-1)}{p}    }   .
 \end{equation}
Combining (\ref{eq:agreq12a}) and (\ref{eq:agreq14}) and rearranging further we have
\begin{eqnarray}
\begin{array}{lrclcl}
  & \frac{(p-1)}{\sqrt{p}\hat{\gamma}}    & = &  \frac{2(u^2-p) \hat{z} - 2}{u}  \\
\Longleftrightarrow \quad $ $ & \frac{(p-1)^2}{p\hat{\gamma}^2}  & = & \lp \frac{2(u^2-p) \hat{z} - 2}{u}  \rp^2  \\
\Longleftrightarrow \quad $ $ & \frac{(p-1)^2}{p\hat{\gamma}^2}  & = & 4(\hat{z}^2(u^2-p)^2-2\hat{z}(u^2-p) +1 )/u^2  \\
\Longleftrightarrow \quad $ $ & \frac{(p-1)^2}{p\hat{\gamma}^2}  & = & 4((1/p-\hat{z}(u^2-2))(u^2-p)-2\hat{z}(u^2-p) +1 )/u^2  \\
\Longleftrightarrow \quad $ $ & \frac{(p-1)^2}{p\hat{\gamma}^2}  & = & 4((1/p-\hat{z}u^2)(u^2-p) +1 )/u^2 \\
\Longleftrightarrow \quad $ $ & \frac{(p-1)^2}{p\hat{\gamma}^2}  & = & 4(-\hat{z}(u^2-p) +1/p)  \\
\Longleftrightarrow \quad $ $ & \frac{(p-1)^2}{4\hat{\gamma}^2}  & = & (-\hat{z}(u^2-p)p +1)  \\
\Longleftrightarrow \quad $ $ & \frac{1}{4\hat{\gamma}^2}  & = & (-\hat{z}(u^2-p)p +1)/(p-1)^2 ,
\end{array}\label{eq:agreq15}
\end{eqnarray}
where the fourth equality follows from (\ref{eq:agreq6}). We also observe the following sequence of equalities
\begin{eqnarray}
\begin{array}{lrclcl}
  & 0  & = &    ( -\hat{z}^2(u^2-p)p -\hat{z}(u^2-2)p +1)  \\
\Longleftrightarrow \quad $ $ &  \hat{z}(p^2-2p)  & = &  (-\hat{z}^2(u^2-p)p -\hat{z}(u^2-p)p +1)  \\
\Longleftrightarrow \quad $ $ &  \hat{z}(p-1)^2   & = &   (-\hat{z}^2(u^2-p)p -\hat{z}(u^2-p)p +\hat{z}+1)  \\
\Longleftrightarrow \quad $ $ &  \hat{z}(p-1)^2   & = &   (\hat{z}+1)(-\hat{z}(u^2-p)p +1) \\
\Longleftrightarrow \quad $ $ &   \hat{z}/(\hat{z}+1)     & = &    (-\hat{z}(u^2-p)p +1)/(p-1)^2,
\end{array}\label{eq:agreq16}
\end{eqnarray}
where the first equality follows from (\ref{eq:agreq6}). Connecting last equalities in (\ref{eq:agreq15}) and (\ref{eq:agreq16}) proves (\ref{eq:agreq10}).

\vspace{.2in}

\noindent\underline{\emph{Proof of (\ref{eq:agreq10a}):}} Starting with a combination of (\ref{eq:agreq12}) and (\ref{eq:agreq14})
 we find
 \begin{eqnarray}
\begin{array}{lrclcl}
  &  1/((-2 \hat{c}_3 (p-1)-2u)/4/(p-1))    & = &      (2\hat{z}(u^2-p)-2)/u  \\
\Longleftrightarrow \quad $ $ &  1/((-2\hat{c}_3 (p-1)-2u))   & = &   (2\hat{z}(u^2-p)-2)/u/4/(p-1) \\
\Longleftrightarrow \quad $ $ & ((-2 \hat{c}_3 (p-1)))     & = &   2u+1/((2\hat{z}(u^2-p)-2)/u/4/(p-1))  \\
\Longleftrightarrow \quad $ $ & -\hat{c}_3   & = &     u/(p-1)+1/((2\hat{z}(u^2-p)-2)/u/2)   \\
\Longleftrightarrow \quad $ $ &   -\hat{c}_3 u    & = &   u^2/(p-1)+2u^2/((2\hat{z}(u^2-p)-2)) .
\end{array}\label{eq:agreq17}
  \end{eqnarray}
We then have the following sequence of equalities as well
 \begin{eqnarray}
\begin{array}{lrclcl}
  &   -(-\hat{z}(u^2)(p^2-2p+1 )  )      & = &    \hat{z}u^2 (1-2p+p^2)  \\
\Longleftrightarrow \quad $ $ &  -(-\hat{z}(u^2)(p^2-p) + \hat{z}(u^2)(p-1) )   & = &  \hat{z} u^2 (1-2p+p^2) \\
\Longleftrightarrow \quad $ $ &  -(-\hat{z}(u^2-2)(p^2-p) + \hat{z}(u^2-2p)(p-1) )      & = &  \hat{z}u^2 (1- 2p +p^2)  \\
\Longleftrightarrow \quad $ $ &   -(-z\hat{}(u^2-2)(p^2-p) + \hat{z}(u^2-2p)(p-1) )    & = &  \hat{z} u^2- u^2(-\hat{z}p(-2))-u^2 p \hat{z}(-p)  \\
\Longleftrightarrow \quad $ $ &   -(-z\hat{}(u^2-2)(p^2-p) + \hat{z}(u^2-2p)(p-1) )    & = &  \hat{z}u^2- u^2(-\hat{z}p(u^2-2))-u^2p\hat{z}(u^2-p) .
 \end{array}\nonumber \\ \label{eq:agreq18}
  \end{eqnarray}
Focusing on the left hand side of the last equality  we further find
 \begin{eqnarray}
 L_s & \triangleq &
  -(-z\hat{}(u^2-2)(p^2-p) + \hat{z}(u^2-2p)(p-1) )   \nonumber \\
  & = &    -((1/p-\hat{z}(u^2-2))(p^2-p) + \hat{z}(u^2-2p)(p-1) -(p-1))
  \nonumber \\
  & = &   -((\hat{z}^2(p^2-p)(u^2-p) + \hat{z}(u^2-2p)(p-1) -(p-1)))
  \nonumber \\
  & = &    -((\hat{z}^2(p^2-p)(u^2-p)-(p^2-p)\hat{z} + \hat{z}(u^2-p)(p-1) -(p-1)))
  \nonumber \\
  & = &   -(p-1+(p^2-p)\hat{z})((\hat{z}(u^2-p)-1))
  \nonumber \\
  & = &   -(1+p\hat{z})(p-1)((\hat{z}(u^2-p)-1)) .
 \label{eq:agreq19}
  \end{eqnarray}
For the right hand side of the last equality in (\ref{eq:agreq18}) we have
 \begin{eqnarray}
R_s & \triangleq &
\hat{z}u^2- u^2(-\hat{z}p(u^2-2))-u^2p\hat{z}(u^2-p)   \nonumber \\
  & = &   (\hat{z}+1)u^2- u^2(1-\hat{z}p(u^2-2))-u^2p\hat{z}(u^2-p)
  \nonumber \\
  & = &  (\hat{z}+1)u^2- u^2 p\hat{z}^2(u^2-p)-u^2 p\hat{z}(u^2-p)
  \nonumber \\
  & = &  (z\hat{}+1)u^2-(\hat{z}+1) u^2 p\hat{z}(u^2-p)
  \nonumber \\
  & = &   (\hat{z}+1)u^2(1-p\hat{z}(u^2-p))
  \nonumber \\
  & = &   p(\hat{z}+1)u^2 (1/p-\hat{z}(u^2-p))
  \nonumber \\
  & = &  p(\hat{z}+1)u^2(1/p-\hat{z}(u^2-2)+\hat{z}(p-2))
  \nonumber \\
  & = &  p(\hat{z}+1)u^2(\hat{z}^2(u^2-p)+\hat{z}(p-2))
  \nonumber \\
  & = &  p\hat{z}(\hat{z}+1)u^2(\hat{z}(u^2-p)+p-2).
 \label{eq:agreq20}
  \end{eqnarray}
Combining (\ref{eq:agreq19}) and (\ref{eq:agreq20}) with the last equality of (\ref{eq:agreq18}) we further have
 \begin{eqnarray}
\begin{array}{lrclcl}
  & \hspace{-.3in}  -(-z\hat{}(u^2-2)(p^2-p) + \hat{z}(u^2-2p)(p-1) )    & = &  \hat{z}u^2- u^2(-\hat{z}p(u^2-2))-u^2p\hat{z}(u^2-p)   \\
\Longleftrightarrow \quad $ $ &  L_s  & = & R_s\\
\Longleftrightarrow \quad $ $ &   -(1+p\hat{z})(p-1)((\hat{z}(u^2-p)-1))    & = & p\hat{z}(\hat{z}+1)u^2(\hat{z}(u^2-p)+p-2) \\
\Longleftrightarrow \quad $ $ &   -(1+p\hat{z})/p/\hat{z}(p-1)((\hat{z}(u^2-p)-1))  & = &    (\hat{z}+1)u^2(\hat{z}(u^2-p)+p-2)  \\
\Longleftrightarrow \quad $ $ &   -(1+p\hat{z})/p/\hat{z}(p-1)((\hat{z}(u^2-p)-1))  & = &    (\hat{z}+1)u^2((\hat{z}(u^2-p)-1)+p-1)  \\
\Longleftrightarrow \quad $ $ &  -(1+p\hat{z})/p/\hat{z}   & = & \hspace{-.02in} (\hat{z}+1)u^2/(p-1)+(\hat{z}+1)u^2/((\hat{z}(u^2-p)-1))  \\
\Longleftrightarrow \quad $ $ &   -(1+p\hat{z})/p/\hat{z}/(1+\hat{z})   & = &    u^2/(p-1)+u^2/((\hat{z}(u^2-p)-1))\\
\Longleftrightarrow \quad $ $ &   -(1+p\hat{z})/p/\hat{z}/(1+\hat{z})   & = &  u^2/(p-1)+2u^2/((2\hat{z}(u^2-p)-2)).
  \end{array}\nonumber \\ \label{eq:agreq21}
  \end{eqnarray}
Connecting beginning and end in (\ref{eq:agreq18}) with the beginning and end in (\ref{eq:agreq21}) we obtain
 \begin{eqnarray}
\begin{array}{lrclcl}
  &   -(-\hat{z}(u^2)(p^2-2p+1 )  )      & = &    \hat{z}u^2 (1-2p+p^2)  \\
 \Longleftrightarrow \quad $ $ &   -(-z\hat{}(u^2-2)(p^2-p) + \hat{z}(u^2-2p)(p-1) )    & = &  \hat{z}u^2- u^2(-\hat{z}p(u^2-2))-u^2p\hat{z}(u^2-p) \\
\Longleftrightarrow \quad $ $ &   -(1+p\hat{z})/p/\hat{z}/(1+\hat{z})   & = &  u^2/(p-1)+2u^2/((2\hat{z}(u^2-p)-2)). \end{array} \label{eq:agreq22}
  \end{eqnarray}
A combination of the last equality in (\ref{eq:agreq17}) and the last equality in (\ref{eq:agreq22}) completes the proof of (\ref{eq:agreq10a}).

\vspace{.3in}

\noindent \underline{\emph{\textbf{Concrete results  based on Theorem \ref{thm:thm3}}}}

Results obtained after numerical evaluations are shown in Figure \ref{fig:fig1}. For several smallest $p$ the concrete GSE values are also shown in Table \ref{tab:tab1} (they are denoted by $\xi^{(2,p)}_{sph}(p) $ and they match (after $\sqrt{p}$ scaling) the numerical values given for upper bounds in \cite{DarMc24}). In parallel, we show the results that the above mechanism produces if $c_3\rightarrow 0$
(they are denoted by $\xi^{(1)}_{sph}(p) $ and they match the replica symmetric predictions). Moreover, both these values also correspond to what is called first and second partial level of lifting within the fully lifted (fl) RDT \cite{Stojnicflrdt23,Stojnicnflgscompyx23}.

\begin{figure}[h]
\centering
\centerline{\includegraphics[width=1.00\linewidth]{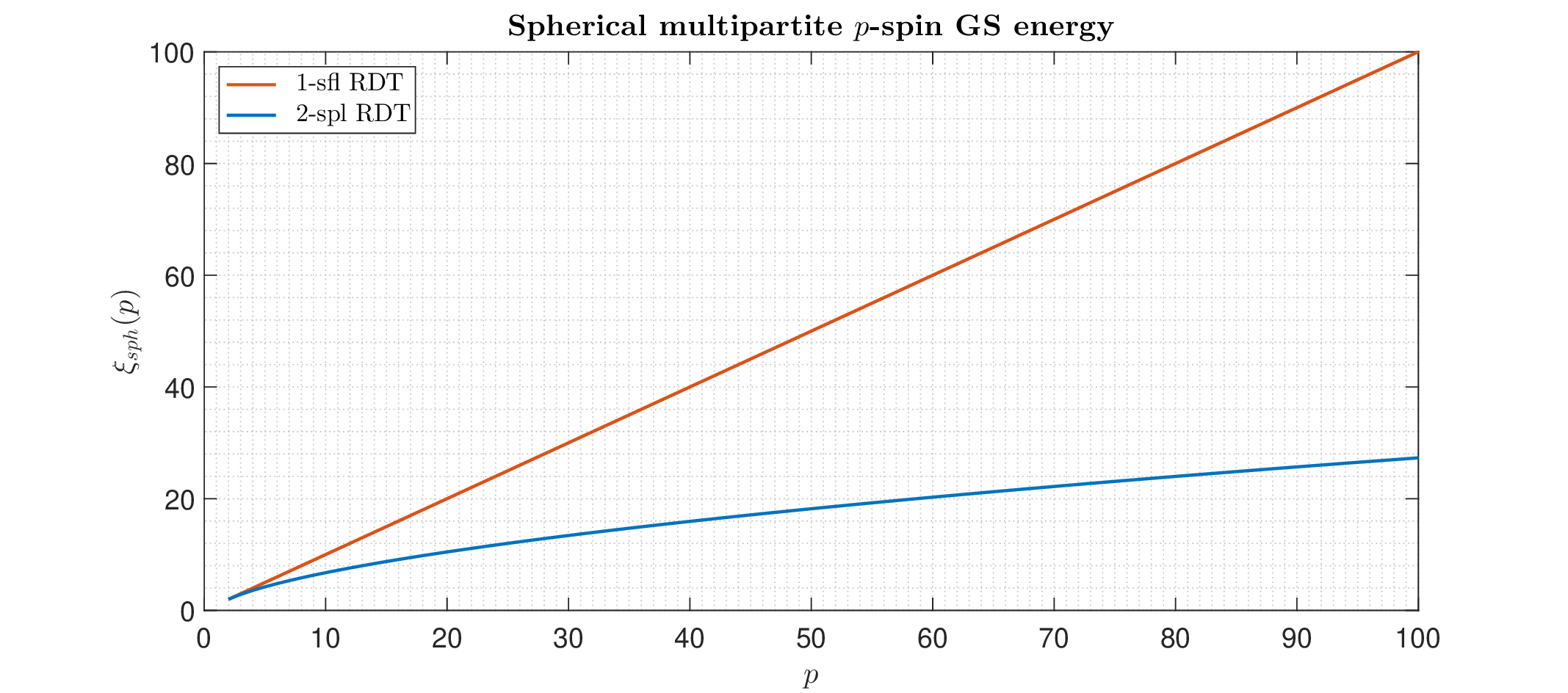}}
\caption{$\xi_{sph}(p)$ as a function of $p$}
\label{fig:fig1}
\end{figure}

\begin{table}[h]
\caption{Spherical multipartite $p$-spin  GSE -- concrete values  for different $p$ }\vspace{.1in}
\centering
\def\arraystretch{1.2}
\begin{tabular}{||l||c||c||c||c||c||c|| }\hline\hline
 \hspace{-0in}$p$                                             & $2$    & $3$ & $4$ &  $5$ & $6$ &  $7$   \\ \hline\hline
$\xi_{sph}^{(1)} (p )$                                         &  \bl{$\mathbf{2.0000}$} &  \bl{$\mathbf{3.0000}$} &  \bl{$\mathbf{4.0000}$} & \bl{$\mathbf{5.0000}$}  &  \bl{$\mathbf{6.0000}$} &  \bl{$\mathbf{7.0000}$} \\ \hline\hline
$\xi_{sph}^{(2,p)} (p )$                                         & \bl{$\mathbf{2.0000}$}  & \bl{$\mathbf{2.8700}$}  & \bl{$\mathbf{3.5882}$} &
\bl{$\mathbf{4.2217}$}  & \bl{$\mathbf{4.7977}$} & \bl{$\mathbf{5.3311}$}  \\ \hline\hline
\end{tabular}
\label{tab:tab1}
\end{table}

\subsection{Ising specialization}
\label{sec:ldpising}

\begin{theorem}
  \label{thm:thm4}
 Assume the setup of Theorem \ref{thm:thm1} and Corollary \ref{cor:cor1}  with $\cS^{(j)} = \cS = \mB^n$. Let $\hat{c}_3$ be such that
 \begin{eqnarray}
 \label{eq:thmeq00a}
\hat{c}_3 = \mbox{argmin}_{c_3} \lp  \frac{c^2_{3}}{2}  + \log  \lp  \erfc  \lp -\frac{c_3\sqrt{p}}{\sqrt{2}} \rp -c_3 u \rp
 \rp.
\end{eqnarray}
Then 
 \begin{eqnarray}
 \label{eq:thmeq00}
  \frac{1}{ \erfc  \lp -\frac{\hat{c}_3\sqrt{p}}{\sqrt{2}} \rp }  = e^{\frac{\hat{c}_3^2p}{2}}( u - \hat{c}_3 )\sqrt{\frac{\pi}{2p}},
  \end{eqnarray}
and with
 \begin{eqnarray}
 \label{eq:thm4eq0}
 \phi_{\mB^n}(p,u)
   \triangleq
 - \frac{\hat{c}_{3}^2}{2}(p-1)  -   \log  \lp u-\hat{c}_3 \rp
- \hat{c}_3 u -\frac{1}{2}\log\lp \frac{\pi}{2p}\rp
=
  \frac{\hat{c}_{3}^2}{2}  +  \log  \lp  \erfc  \lp -\frac{\hat{c}_3\sqrt{p}}{\sqrt{2}} \rp  \rp -\hat{c}_3 u,
 \end{eqnarray}
one has for $\zeta(\cdot)$ from  (\ref{eq:ldpeq1}) 
 \begin{eqnarray}
 \label{eq:thm4eq1}
 \lim_{n\rightarrow \infty }  \frac{1}{n} \log \lp \mP \lp \zeta(p;\mB^n,n) \geq u \rp \rp  \leq \phi_{\mB^n}(p,u) .
  \end{eqnarray}
\end{theorem}

\begin{proof}
As in the proof or Theorem \ref{thm:thm3},  Chernoff/Markov inequality gives
 \begin{eqnarray}
 \label{eq:ldpeq20}
\frac{1}{n}\log \lp  \mP \lp \zeta(p;\mB^n,n) \geq u \rp \rp
 \leq  \frac{1}{n} \log \lp e^{c_3 \sqrt{n} \zeta(p;\mB^n,n) -c_3\sqrt{n} u} \rp
 =  \frac{1}{n} \log \lp e^{c_3 \sqrt{n} \zeta(p;\mB^n,n) } \rp -\frac{1}{n}c_3\sqrt{n} u .
  \end{eqnarray}
Another combination of (\ref{eq:ldpeq2})  with Corollary \ref{cor:cor2} together with $c_3\rightarrow c_3\sqrt{n}$ scaling allows to write the following Ising analogue to (\ref{eq:ldpeq3})
 \begin{eqnarray}
 \label{eq:ldpeq21}
\frac{1}{n}\log \lp  \mP \lp \zeta(p;\mB^n,n) \geq u \rp \rp
& \leq &
 -\frac{c_{3}^2}{2}(p-1)  +  \frac{1}{n}  \log\lp  \mE_{\g} e^{c_{3}\sqrt{n}\sqrt{p}\max_{\x\in\mB^n} \g^T\x } \rp
-c_3 u
\nonumber \\
& = &
  -\frac{c_{3}^2}{2}(p-1)  +   \log\lp  \mE_{\g} e^{c_{3}\sqrt{p} |\g_1|  } \rp
-c_3 u\nonumber \\
& = &
  \frac{c_{3}^2}{2}  +   \log  \lp  \erfc  \lp -\frac{c_3\sqrt{p}}{\sqrt{2}} \rp \rp
-c_3 u
\nonumber \\
& = & \phi_3(c_3).
  \end{eqnarray}
  After computing the derivative one finds
 \begin{eqnarray}
 \label{eq:ldpeq22}
  \frac{d\phi_3(c_3)}{dc_3} = c_3 + \frac{\sqrt{2p}}{\sqrt{\pi}}\frac{e^{-\frac{c_3^2p }{2}}}{ \erfc  \lp -\frac{c_3\sqrt{p}}{\sqrt{2}} \rp }  - u =0 .
  \end{eqnarray}
Taking $\hat{c}_3$ such that $  \frac{d\phi_3(c_3)}{dc_3} =0$, one from (\ref{eq:ldpeq22}) obtains
 \begin{eqnarray}
 \label{eq:ldpeq23}
  \frac{1}{ \erfc  \lp -\frac{\hat{c}_3\sqrt{p}}{\sqrt{2}} \rp }  = e^{\frac{\hat{c}_3^2p}{2}}( u - \hat{c}_3 )\sqrt{\frac{\pi}{2p}}.
  \end{eqnarray}
 A combination of (\ref{eq:ldpeq21}) and  (\ref{eq:ldpeq23}) further gives
 \begin{equation}
 \label{eq:ldpeq24}
  \phi_3(\hat{c}_3)
  =
  \frac{\hat{c}_{3}^2}{2}  -   \log  \lp  \frac{1}{\erfc  \lp -\frac{\hat{c}_3\sqrt{p}}{\sqrt{2}} \rp } \rp
- \hat{c}_3 u
   =
 - \frac{\hat{c}_{3}^2}{2}(p-1)  -   \log  \lp u-\hat{c}_3 \rp
- \hat{c}_3 u -\frac{1}{2}\log\lp \frac{\pi}{2p}\rp
 = \phi_{\mB^n}(p,u),
  \end{equation}
  which completes the proof.
\end{proof}

Proceeding as after Theorem \ref{thm:thm3}, we set
 \begin{eqnarray}
 \label{eq:ldpeq25}
u_{GS}^{(sk)}   \triangleq
\min & & u \nonumber \\
 \mbox{such that} & &  \phi_{\mB^n}(p,u) <0,
  \end{eqnarray}
or alternatively
 \begin{eqnarray}
 \label{eq:ldpeq25}
u_{GS}^{(sk)}   \triangleq
\min_{c_3>0} \lp  \frac{c_{3}}{2}  +  \frac{1}{c_3} \log  \lp  \erfc  \lp -\frac{c_3\sqrt{p}}{\sqrt{2}} \rp \rp
 \rp.
  \end{eqnarray}
This then
gives for the GSE of the Ising $p$-spin model
 \begin{eqnarray}
 \label{eq:ldpeq26}
\xi^0_{sk}(p) \leq u_{GS}^{(sk)},
   \end{eqnarray}
   and
   for the GSE of the Ising multipartite $p$-spin model
 \begin{eqnarray}
 \label{eq:ldpeq27}
\xi_{sk}(p) \leq  \sqrt{p}u_{GS}^{(sk)}.
   \end{eqnarray}
Provided that (\ref{eq:ldpeq26}) holds with equality, one has equality in (\ref{eq:ldpeq27}) as well (numerical optimization over Parisi functional suggests that his might be the case at least for some $p$).

\vspace{.2in}

\noindent \underline{\emph{\textbf{Concrete results based on Theorem \ref{thm:thm4}}}}

Evaluating the above given GSE (bounds) one obtains results shown in Figure \ref{fig:fig2}.  Concrete values obtained for several smallest $p$ are also shown in Table \ref{tab:tab2} (analogously to the spherical case, they are denoted by $\xi^{(2,p)}_{sk}(p) $). The corresponding replica symmetry predictions  (denoted by $\xi^{(1)}_{sk}(p)$ and obtained by the above mechanism for $c_3\rightarrow 0$) are shown as well. As in the above discussed spherical scenario, these values also correspond to the first and second partial level of lifting within the fl RDT \cite{Stojnicflrdt23,Stojnicnflgscompyx23}.

\begin{figure}[h]
\centering
\centerline{\includegraphics[width=1.00\linewidth]{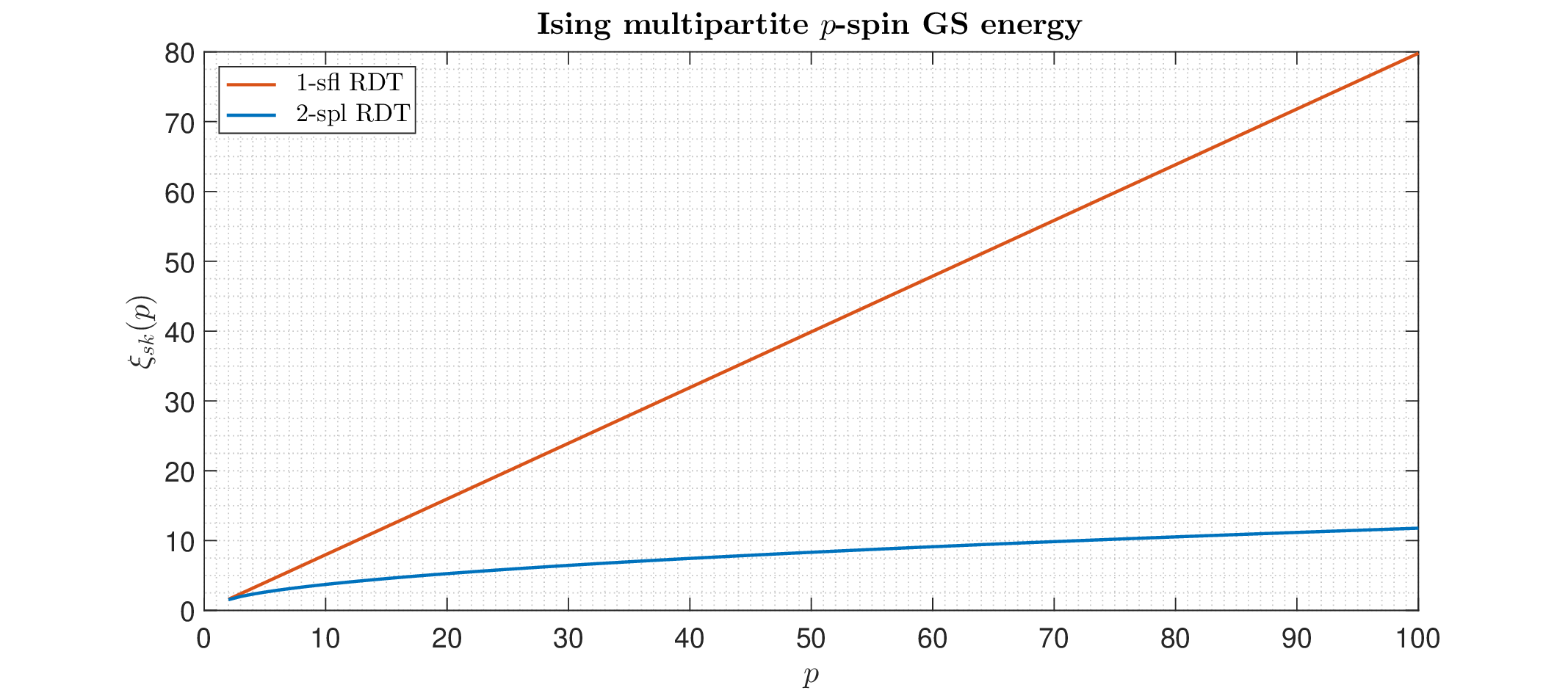}}
\caption{$\xi_{sk}(p)$ (bound) as a function of $p$}
\label{fig:fig2}
\end{figure}

\begin{table}[h]
\caption{Ising multipartite $p$-spin  GSE  -- concrete  values (bounds)  for different $p$ }\vspace{.1in}
\centering
\def\arraystretch{1.2}
\begin{tabular}{||l||c||c||c||c||c||c|| }\hline\hline
 \hspace{-0in}$p$                                             & $2$    & $3$ & $4$ &  $5$ & $6$ &  $7$   \\ \hline\hline
$\xi_{sk}^{(1)} (p )$                                         &  \bl{$\mathbf{1.5958}$} &  \bl{$\mathbf{  2.3937}$} &  \bl{$\mathbf{3.1915}$} & \bl{$\mathbf{3.9894}$}  &  \bl{$\mathbf{4.7873}$} &  \bl{$\mathbf{5.5852}$} \\ \hline\hline
$\xi_{sk}^{(2,p)} (p )$                                         & \bl{$\mathbf{1.5377}$}  & \bl{$\mathbf{1.9927}$}  & \bl{$\mathbf{ 2.3348}$} &
\bl{$\mathbf{ 2.6235}$}  & \bl{$\mathbf{2.8796}$} & \bl{$\mathbf{3.1130}$}  \\ \hline\hline
\end{tabular}
\label{tab:tab2}
\end{table}

\subsection{Discussion}
\label{sec:disc}

Studying general multipartite $p$-spin systems can be done via fully lifted (fl) RDT \cite{Stojnicflrdt23,Stojnicnflgscompyx23}.   However, fl RDT to a large degree relies on heavy numerical evaluations even when $p=2$. The numerical burden exponentially increases with $p$ making it fairly unpractical already for $p\sim 5-10$ when quick assessments are needed. Things do simplify a bit when $\cS^{(j)}=S,1\leq j\leq p$, but the need for quicker more practical characterization remains. The bounds that we presented above are particularly useful in such contexts as they are substantially simpler. To have them fully operational, an assessment of their accuracy is welcome as well. The fact that there are scenarios where they  provide the exact characterizations is both convenient and intriguing. Logical continuation may ask to what extent one can actually rely on such exactness. To formulate possible directions that might lead to addressing this question, we conveniently state the following corollary (an immediate consequence of Corollary \ref{cor:cor1}).

\begin{corollary}
For even $p>2$  let $pMP(\bcS)$ denote the multipartite $p$-spin models where $j$-th $n$-dimensional spin vector, $\x^{(j)}$, belongs to set $\cS^{(j)}\subseteq\mS^n$ and let $pSP(\cS)$ denote the corresponding single-partite $p$-spin models where single $n$-dimensional spin vector, $\x$,  belongs to set $\cS\subseteq\mS^n$. If
   \begin{eqnarray}
   \label{eq:cor3eq1}
   & & \hspace{-.25in} \star \quad  \cS^{(j)}=\cS, \quad (\mbox{i.e., the structures of the sets to which the spins from different groups belong are identical})
\nonumber \\
   & & \hspace{-.25in}  \star \quad  GSE(pSP(\cS)) \quad \mbox{is achieved on the second partial level of lifting},\nonumber
   \end{eqnarray}
   then
   \begin{eqnarray}
   \label{eq:cor3eq2}
    GSE(pSP(\bcS))= \sqrt{p} GSE(pSP(\cS)) .
   \end{eqnarray}
   \label{cor:cor3}
\end{corollary}

\begin{proof}
Automatically follows from  Corollary \ref{cor:cor1} as the needed ``sandwiching'' bounding mechanism is provided there.
\end{proof}

\begin{remark}
The $p$-evenness is likely not needed in the above corollary (and also anywhere else in the paper). Examining Talagrand's positivity principle \cite{Talpos00} and combining the  mechanisms introduced here together  with Panchenko's considerations from \cite{Pan10,Pan10a,Pan13,Pan13a,Pan13b} provides viable avenues towards removing this restriction.
\end{remark}

As we have seen, the condition of the above corollary is satisfied for the spherical sets and a strong numerical evidence suggests that it might also be true for the Ising sets. En route to settling compressed sensing large deviations,  \cite{Stojnicl1RegPosasym} obtained that the LDP considerations similar to the ones presented above hold for more general convex sets. This imminently suggests the following

\vspace{-.0in}\begin{center}
 \begin{tcolorbox}[title= Two interesting questions: ]
\vspace{-.15in}
 \begin{eqnarray*}
& &  \hspace{-.2in} \mbox{ 1) \emph{What are the general characteristics of sets $\cS$ for which the conditions of Corollary \ref{cor:cor3} hold?  }}
\\
& & \hspace{-.2in} \mbox{ 2) \emph{What is the ovperall role of convexity with respect to the conditions of Corollary \ref{cor:cor3}? }}
  \end{eqnarray*}
\vspace{-.2in}
 \end{tcolorbox}
\end{center}\vspace{-.0in}

\noindent Keeping in mind that Ising set $\mB^n$ is actually obtained as the intersection of the convex set $\{\x|\x_i^2\leq\frac{1}{n}\}$  and the unit sphere $\mS^n$, a very interesting next step might be to consider the intersection of $\{\x|\x_i^2\leq\frac{c}{n},c>1\}$ and $\mS^n$.

\begin{singlespace}
\bibliographystyle{plain}
\bibliography{nflgscompyxRefs}
\end{singlespace}

\end{document}